\documentclass[a4paper,11pt]{article}

\usepackage[utf8]{inputenc}
\usepackage[british]{babel}

\usepackage[mono=false]{libertine}
\usepackage[T1]{fontenc}

\usepackage{amsthm}
\usepackage[cmintegrals,libertine]{newtxmath}
\usepackage[cal=euler, scr=boondoxo]{mathalfa}
\useosf

\usepackage{microtype}

\usepackage{numprint}

\usepackage[margin=2.5cm]{geometry}
\usepackage{multirow}

\usepackage{tikz}

\usepackage[font=sf]{caption}
\usepackage{hyperref}
\hypersetup{
	colorlinks=true,
		citecolor=blue!60!black,
		linkcolor=red!60!black,
		urlcolor=green!40!black,
		filecolor=yellow!50!black,
	breaklinks=true,
	pdfpagemode=UseNone,
	bookmarksopen=false,
}

\usepackage{enumitem}

\newtheorem{thm}{Theorem}
\newtheorem{lem}{Lemma}
\newtheorem{prop}{Proposition}
\newtheorem{cor}{Corollary}

\theoremstyle{definition}
\newtheorem{rem}{Remark}

\newcommand{\ensembles}[1]{\mathbb{#1}}
	
	\newcommand{\Z}{\ensembles{Z}}

\newcommand{\ind}[1]{\mathbf{1}_{\{#1\}}}
	\renewcommand{\P}{\ensembles{P}}
	\newcommand{\E}{\ensembles{E}}

\renewcommand{\Pr}[1]{\P\left(#1\right)}
\newcommand{\Prc}[2]{\P\left(#1 \;\middle|\; #2\right)}
\newcommand{\Es}[1]{\E\left[#1\right]}

\newcommand{\ex}{\mathrm{e}}

\newcommand{\q}{\mathbf{q}}
\newcommand{\e}{\overline{\mathfrak{e}}}

\newcommand{\dgr}{d_{\mathrm{gr}}}

\newcommand{\aper}{\mathrm{aper}}
\newcommand{\T}{\mathcal{T}}

\newcommand{\A}{\mathcal{A}}

\newcommand{\Ball}{\mathrm{Ball}}
\newcommand{\hBall}{\overline{\mathrm{Ball}}}

        \newcommand{\map}{\mathfrak{m}}

\newcommand{\rootface}{f_{ \mathrm{r}}}

        \newcommand{\Map}{\mathfrak{M}}
        \newcommand{\Tree}{\mathscr{T}}

\newcommand{\cvdist}[1][n]{\enskip\xrightarrow[#1 \to \infty]{(d)}\enskip}

\DeclareSymbolFont{extraup}{U}{zavm}{m}{n}
\DeclareMathSymbol{\vardspade}{\mathalpha}{extraup}{81}
\DeclareMathSymbol{\varheart}{\mathalpha}{extraup}{86}
\DeclareMathSymbol{\vardiamond}{\mathalpha}{extraup}{87}
\DeclareMathSymbol{\varclub}{\mathalpha}{extraup}{84}

\makeatletter
\renewcommand*{\@fnsymbol}[1]{\ensuremath{\ifcase#1\or  \vardspade \or \vardiamond \or \varheart\or \varclub \or
   \mathsection\or \mathparagraph\or \|\or **\or \dag\dag
   \or \ddagger\ddagger \else\@ctrerr\fi}}
\makeatother

\title{Markovian explorations of random planar maps are roundish}

\author{
	Nicolas \textsc{Curien}\thanks{D\'epartement de Math\'ematiques, Univ. Paris-Sud, Universit\'e Paris-Saclay and IUF.\hfill  \href{mailto:nicolas.curien@gmail.com}{\texttt{nicolas.curien@gmail.com}}} 
\qquad\&\qquad
	Cyril \textsc{Marzouk}
	\thanks{CNRS, IRIF UMR 8243, Universit\'{e} Paris-Diderot, France.\hfill  \href{mailto:cmarzouk@irif.fr}{\texttt{cmarzouk@irif.fr}}}
}

\begin{document}

\maketitle

\begin{abstract}
The infinite discrete stable Boltzmann maps are ``heavy-tailed''  generalisations of the well-known Uniform Infinite Planar Quadrangulation. Very efficient tools to study these objects are Markovian step-by-step explorations of the lattice called peeling processes. Such a process depends on an algorithm which selects at each step the next edge where the exploration continues. We prove here that, whatever this algorithm, a peeling process always reveals about the same portion of the map, thus growing roughly metric balls. Applied to well-designed algorithms, this easily enables us to compare distances in the map and in its dual, as well as to control the so-called pioneer points of the simple random walk, both on the map and on its dual.
\end{abstract}

\section{Introduction}

\paragraph{Peeling process.} Since the introduction of the UIPT (Uniform Infinite Planar Triangulation) by Angel \& Schramm~\cite{Angel-Schramm:UIPT}, the study of the large scale properties of infinite planar maps has been an intensive field of research. One of the main tools to study these random graphs, is the so-called \emph{peeling process} which is a Markovian way to explore these random discrete surfaces step-by-step and connects them with random walks. It has fruitfully been used in the case of the UIPT (or its quadrangular cousin the UIPQ) to study e.g. its geometry~\cite{Angel:Growth_and_percolation_on_the_UIPT, Curien-Le_Gall:Scaling_limits_for_the_peeling_process_on_random_maps}, the behaviour of the simple random walk~\cite{Benjamini-Curien:Simple_random_walk_on_the_uniform_infinite_planar_quadrangulation_subdiffusivity_via_pioneer_points}, its conformal structure~\cite{Curien-A_glimpse_of_the_conformal_structure_of_random_planar_maps} or fine properties of percolation~\cite{Angel:Growth_and_percolation_on_the_UIPT, Angel-Curien:Percolations_on_random_maps_half_plane_models}. The idea of Markovian exploration of random triangulations can be traced back to Watabiki in the physics literature and it was formalised first by Angel~\cite{Angel:Growth_and_percolation_on_the_UIPT}. Later Budd~\cite{Budd:The_peeling_process_of_infinite_Boltzmann_planar_maps} introduced a different and more robust version of it which we will use below.

The main advantage of the peeling process is the flexibility on the choice of the exploration, which depends on an \emph{algorithm}, and the results cited above were obtained using different peeling algorithms. However, certain properties of the peeling process are universal in the sense that they do not depend upon the algorithm: for example the law of the underlying random walk driving the perimeter process, or the fact that any peeling algorithm eventually discovers the complete underlying lattice~\cite[Corollary~6]{Curien-Le_Gall:Scaling_limits_for_the_peeling_process_on_random_maps}. In this work we will show, in a rather strong sense, that all Markovian explorations of the UIPT/UIPQ are bound to discover roughly the same portion of the map at time $n$. In fact, our result applies more generally to infinite (bipartite) ``discrete stable'' Boltzmann maps whose definition we now recall.

\paragraph{Infinite Boltzmann map and the filled-in peeling process.} As usual, all planar maps in this work are rooted, i.e. equipped with a distinguished oriented edge, and as it is customary, we will only consider \emph{bipartite} planar maps (all faces have even degree). Given a non-zero sequence $ \mathbf{q}= (q_{k})_{ k \geq 1}$ of non-negative numbers we define the Boltzmann measure $w$ on the set of all bipartite planar maps by the formula
\[w(\map) \coloneqq \prod_{f\in \mathrm{Faces}(\map)} q_{ \mathrm{deg}(f)/2}.\]
We shall assume that $\q$ is a \emph{critical weight sequence of type} $ a \in (\frac{3}{2}, \frac{5}{2}]$. This means that the equation $z  =1 + \sum_{i \geq 1} \binom{2i-1}{i-1} q_{i} z^i$ has a unique solution $Z_\q > 1$ and that the probability measure 
\[\mu(0)=Z_\q^{-1} \qquad \text{and} \qquad  \mu(k)=Z_\q^{k-1}\binom{2k-1}{k-1}q_{k}.\]
has mean one, and either it has finite variance in the case $a = \frac{5}{2}$, or it is in the strict domain of attraction of an $(a-1)$-stable distribution i.e. $\mu([k,\infty)) \sim c k^{-a+1}$ for some constant $c>0$ as $k \to \infty$, see~\cite{Budd-Curien:Geometry_of_infinite_planar_maps_with_high_degrees,Le_Gall-Miermont:Scaling_limits_of_random_planar_maps_with_large_faces, Borot-Bouttier-Guitter:A_recursive_approach_to_the_O_n_model_on_random_maps_via_nested_loops} and~\cite{Curien-Richier:Duality_of_random_planar_maps_via_percolation} for details. One can then define (using the assumption of criticality only) a random infinite bipartite map  $\Map_\infty$ of the plane as the local weak limit of random maps sampled according to $w(\cdot)$ and conditioned to be large, see~\cite{Bjornberg-Stefansson:Recurrence_of_bipartite_planar_maps,Stephenson:Local_convergence_of_large_critical_multi_type_Galton_Watson_trees_and_applications_to_random_maps}. We will consider so-called \emph{filled-explorations} of $\Map_\infty$ which are sequences of submaps with one hole
\[\e_{0} \subset \cdots \subset \e_{n} \subset \cdots \subset \Map_\infty,\]
obtained by starting with the root-edge $ \e_{0}$ of $\Map_\infty$ and iteratively \emph{peeling} an edge on the boundary of $ \e_{n}$ at each step. If the peeling of an edge creates more than one hole, then we immediately fill-in the finite part (recall that $\Map_\infty$ is one-ended), see Section~\ref{sec:peeling} and~\cite{Curien:Peccot} for details. As recalled above, these explorations depend on an algorithm, hereafter denoted by $\mathcal{A}$ to choose the next edge to peel $ \mathcal{A}( \e_{n})$ on the boundary of the explored part. This algorithm can be deterministic or it may depend on another source of randomness, as long as it is independent of $\Map_\infty$,  and we denote by $(\e\vphantom{e}^\A_n)_{n \geq 0}$ the filled-in  peeling exploration of $ \Map_\infty$ to highlight the dependence in $ \mathcal{A}$. 

The ball of radius $r$ in $ \Map_{\infty}$, denoted by $\Ball( \Map_\infty,r)$, is obtained by keeping the faces of $\Map_\infty$ which have at least one vertex at graph distance smaller than $r$ from the origin $\rho$ of (the root-edge of) the map; its hull $\hBall( \Map_\infty,r)$ is obtained by filling-in all the finite regions of the complement of $\Ball( \Map_\infty,r)$ in $\Map_{\infty}$ (recall that $\Map_{\infty}$ is one-ended). Our main result is then the following, which explains the title:

\begin{thm}
\label{thm:volume_peeling}
Fix a critical weight sequence $  \mathbf{q}$ of type $a \in (\frac{3}{2}, \frac{5}{2}]$. For any $ \varepsilon>0$, there exist  $0 < c_{ \varepsilon} < C_{ \varepsilon}< \infty$ such that for any algorithm $ \mathcal{A}$ we have for every $n$ large enough,
\[\hBall\left(\Map_\infty, c_{ \varepsilon} n^{\frac{1}{2(a-1)}}\right)
\subset \e\vphantom{e}^\A_n
\subset \hBall\left(\Map_\infty, C_{ \varepsilon} n^{\frac{1}{2(a-1)}}\right)\]
with probability at least $1 - \varepsilon$.
\end{thm}

Our main result thus shows, in a sense, that Markovian explorations of $ \Map_\infty$ are bound to discover roughly the same region of $ \Map_\infty$ by time $n$. In particular this implies that any Markovian exploration will eventually discover the full map, a fact already proved in~\cite[Corollary~27]{Curien:Peccot}. In the other direction, the paper~\cite{Curien-Marzouk:How_fast_planar_maps_get_swallowed_by_a_peeling_process} studies the decay of the (small) probability that a given edge remains exposed on the boundary of $ \e_{n}$ for large $n$'s. 

\begin{rem} There is little doubt that our results also hold in the case of non-bipartite maps but we restrict to the case of bipartite maps for technical convenience. In the particular case of the UIPT (type I) our geometric estimates on maps (Proposition~\ref{prop:volume_boules}) can be derived from~\cite{Budzinski:The_hyperbolic_Brownian_plane} and Proposition~\ref{prop:volume_boules_cartes_bord} and~\ref{prop:aperture} on maps with a boundary may be obtained using similar techniques, with~\cite{Miermont:Invariance_principles_for_spatial_multitype_Galton_Watson_trees}. \end{rem} 

Let us now derive corollaries of our main theorem by specifying it to well-chosen peeling algorithms.

\paragraph{Dual graph distance.} There is an algorithm $ \mathcal{L}_{ \mathrm{dual}}$ which explores the hull of the balls for the \emph{dual} metric on $\Map_\infty$ (i.e. the graph distance on the dual graph $\Map_\infty^{\dag}$) whose details can be found in Section~\ref{sec:corollaires}. Since faces of a map correspond to vertices of its dual, in order to compare these two lattices, let us view $\Ball(\Map_\infty^\dag, r)$ as the subset of vertices of $\Map_\infty$ which are incident to a face at dual graph distance from the root-face (the one to the right of the root-edge) less than $r$, then let $\hBall(\Map_\infty^\dag, r)$ be the set of all these vertices to which we add all the finite regions of the complement.

When $a \in (2, \frac{5}{2}]$, the so-called ``dilute phase'', balls in the dual graph grow polynomially in the radius~\cite{Budd-Curien:Geometry_of_infinite_planar_maps_with_high_degrees}; combined with the above result, this yields the following rough comparison between primal and dual distances in $\Map_\infty$: the hull of the ball of radius $r$ in $\Map_\infty^\dag$ is close to the hull of the ball of radius $r^{1/(2a-4)}$ in $\Map_\infty$. For the so-called ``dense phase'' $a \in (\frac{3}{2}, 2)$, the balls in the dual graph grow exponentially in the radius~\cite{Budd-Curien:Geometry_of_infinite_planar_maps_with_high_degrees}, whilst in the intermediate regime $a=2$, they have an ``intermediate growth'', exponential in the \emph{square-root} of the radius~\cite{Budd-Curien-Marzouk:Infinite_random_planar_maps_related_to_Cauchy_processes}, so now the hull of the dual ball of radius $r$ is close to the hull of a primal ball with radius of order $\ex^r$ when $a < 2 $ and $\ex^{\sqrt{r}}$ for $a=2$.

\begin{cor}\label{cor:volume_primal_dual}
Fix a critical weight sequence $\q$ of type $a \in (\frac{3}{2}, \frac{5}{2}]$; there exists $c_a > 0$ such that the following holds: For every $ \varepsilon>0$, there exist $0 < c_{ \varepsilon} < C_{ \varepsilon}< \infty$ such that for every $r$ large enough, we have 
\[\begin{array}{rccclcl}
\hBall(\Map_\infty, c_\varepsilon r^{\frac{1}{2a-4}}) &\subset& \hBall(\Map_\infty^\dag, r) &\subset& \hBall(\Map_\infty, C_\varepsilon r^{\frac{1}{2a-4}})
&\mbox{when}& a \in (2; 5/2],
\\
\hBall(\Map_\infty, \ex^{\pi \sqrt{r/2} (1-\varepsilon)}) &\subset& \hBall(\Map_\infty^\dag, r) &\subset& \hBall(\Map_\infty, \ex^{\pi \sqrt{r/2} (1+\varepsilon)})
&\mbox{when} &a = 2,
\\
\hBall(\Map_\infty, \ex^{c_a (1-\varepsilon) r}) &\subset& \hBall(\Map_\infty^\dag, r) &\subset& \hBall(\Map_\infty, \ex^{c_a  (1+\varepsilon) r})
&\mbox{when}& a \in (3/2; 2),
\end{array}\]
with probability at least $1 - \varepsilon$.
\end{cor}

Observe that $2a-4=1$ when $a = 5/2$; in the case of triangulations, it is known more precisely that the distances on the primal and dual are in fact asymptotically proportional~\cite{Curien-Le_Gall:First_passage_percolation_and_local_modifications_of_distances_in_random_triangulations}.

\paragraph{Pioneer points for simple random walk.} 
In another direction we study the behaviour of the simple random walk on $\Map_\infty$ and $\Map_\infty^{\dag}$ using algorithms $ \mathcal{S}_{ \mathrm{primal}}$ and $ \mathcal{S}_{ \mathrm{dual}}$ respectively which both explore the map $ \Map_\infty$ along the corresponding walk. These algorithms enable us to keep track of the so-called \emph{pioneer points} of the walk, which are roughly speaking steps performed by the walk which yield to the discovery of a new vertex which is not disconnected from infinity when removing the past trajectory (see Section~\ref{sec:pionniers} for details). Our theorem shows that the respective walk performs about
\[\begin{array}{cll}
r^{2a-2} &\quad\text{pioneer steps within } \hBall( \Map_\infty,r)&\quad   \text{(primal)},
\\
g_a(r) &\quad\text{pioneer steps within } \hBall(\Map_\infty^\dag, r)&\quad   \text{(dual)},
\end{array}\]
with high probability, where $g_a(r) = r^{\frac{a-1}{a-2}} \cdot \ind{a > 2} +  \sqrt{\log r}\cdot \ind{a=2} + \log r\cdot \ind{a < 2}$. We refer to Corollary~\ref{cor:points_pionniers_sous_diff} and~\ref{cor:points_pionniers_sous_diff_dual} for more precise statements.

For the walk on the primal map $\Map_\infty$, in the dilute regime $a > 2$, this in particular establishes a \emph{sub-diffusivity} phenomenon in the sense that with high probability, the $n$-th step of the walk is at distance at most $n^{1/(2a-2)}$ from its starting point. This idea to use the pioneer points to derive a sub-diffusive behaviour was exploited in~\cite{Benjamini-Curien:Simple_random_walk_on_the_uniform_infinite_planar_quadrangulation_subdiffusivity_via_pioneer_points} on the UIPQ and considers in some sense the worst case where each step of the random walk is a pioneer point. It is likely that this is far from what really occurs and controlling this would improve the exponent (see~\cite{Curien-Marzouk:How_fast_planar_maps_get_swallowed_by_a_peeling_process} for an argument based on reversibility which improves a tiny bit the exponent in the case of the UIPT/Q). Let us mention that in a forthcoming paper~\cite{Curien-Marzouk:Sous_diff_Boltzmann}, we use a completely different method to prove that the walk on the primal map is actually always sub-diffusive with exponent at most $1/3$, for all $a \in (\frac{3}{2}, \frac{5}{2}]$.

One could apply our main result to many others peeling algorithms such as the uniform peeling (or metric exploration for the Eden model), peeling along percolation interfaces or peeling associated with internal DLA. We refrain from doing so to keep the paper short.\medskip

Throughout this work, for two positive random processes $(X_n)_{n \ge 0}$ and $(Y_n)_{n \ge 0}$, we write $X_n \lesssim Y_n$, resp. $X_n \gtrsim Y_n$, when
\[\lim_{C \to \infty} \limsup_{n \to \infty} \Pr{X_n > C Y_n} = 0,
\qquad\text{resp.}\qquad
\lim_{C \to \infty} \limsup_{n \to \infty} \Pr{X_n < C^{-1} Y_n} = 0.\]
We also write $X_n \approx Y_n$ when both $X_n \lesssim Y_n$ and $X_n \gtrsim Y_n$ hold. This notion of comparison is different from the one used in~\cite{Benjamini-Curien:Simple_random_walk_on_the_uniform_infinite_planar_quadrangulation_subdiffusivity_via_pioneer_points}, where these symbols have the following different meaning: there $X_n \lesssim Y_n$ means that there exists $\kappa > 0$ such that $X_n / (Y_n \log^\kappa n)$ converges almost surely to $0$. This notion is neither weaker nor stronger than the present one (it is a trade-off between a strong convergence and logarithmic factors instead of constants). We believe that all our results also hold in this sense but our current estimates do not imply it.

\paragraph{Acknowledgments} We acknowledge support from the Fondation Math\'{e}matique Jacques Hadamard, the grants \texttt{ERC-2016-STG 716083} ``CombiTop'' and \texttt{ERC 740943} ``GeoBrown'', as well as the grant \texttt{ANR-14-CE25-0014} ``ANR GRAAL''.

\section{Peeling process and geometric estimates}
In this section, we briefly recall the filled-in peeling process of $\Map_{\infty}$ and refer the reader to~\cite{Curien:Peccot} for details. We also gather  the geometric estimates needed for the proof of our main result, which is then rather short and simple. The proofs of these estimates, which are based on constructions with labelled trees, are postponed to Section~\ref{proof:technics}.

\subsection{Filled-in peeling process}\label{sec:peeling}
Given an instance of $\Map_{\infty}$, a filled-in peeling process is a sequence of growing submaps $ \e_{0} \subset \e_{1} \subset \cdots \subset \Map_{\infty}$ where $ \e_{0} $ is the root-edge of $\Map_{\infty}$ and $ \e_{n+1}$ is obtained by the peeling of one edge on the boundary of $ \e_{n}$. More precisely, $ \e_{n}$ is a planar (bipartite) map with a hole, i.e. a distinguished face whose boundary is simple. We say that $ \e_{n}$ is a submap of $ \Map_{\infty}$ in the sense that $\Map_{\infty}$ can be recovered by gluing a proper map with (general) boundary inside the unique hole of $\e_{n}$. To pass on from $ \e_{n}$ to $ \e_{n+1}$ we first select an edge $ \mathcal{A}( \e_{n})$ on the boundary of the hole of $ \e_{n}$, where $ \mathcal{A}$ is our peeling algorithm which may depend on another source of randomness as long as it is independent of $ \Map_{\infty}$. Once $ \mathcal{A}( \e_{n} )$ is picked, we reveal its status inside the map $ \Map_{\infty}$, two cases may appear, as illustrated in Figure~\ref{fig:peeling}:
\begin{itemize} 
\item Either the peel edge is incident to a new face in $ \Map_{\infty}$ of degree $2k$, then $ \e_{n+1}$ is obtained from $ \e _{n}$ by gluing this face on the peel edge without performing any other identification. This event is called event of type $ \mathsf{C}_{k}$.
\item Or the peel edge is incident to another face of $ \e_{n}$ in the map $\Map_{\infty}$, in which case we perform the identification of the two boundary edges of $\e_n$. When doing so, the hole  of $ \e_{n}$ of perimeter, say $2p$, is split into two holes of perimeter $2p_{1}$ and $2 p_{2}$ with $p_{1}+p_{2}=p-1$. Since $\Map_{\infty}$ is one-ended almost surely, only one of these holes corresponds to an infinite region in $\Map_{\infty}$. We then fill-in the finite hole with the corresponding map inside $ \Map_{\infty}$ to obtain $  \e_{n+1}$. We speak of event of type $ \mathsf{G}_{*,p_{1}}$ or $ \mathsf{G}_{p_{2}, *}$ depending whether the finite hole is on the left or on the right of the peel edge.
\end{itemize}

\begin{figure}[!ht]\centering
\includegraphics[width=.85\linewidth]{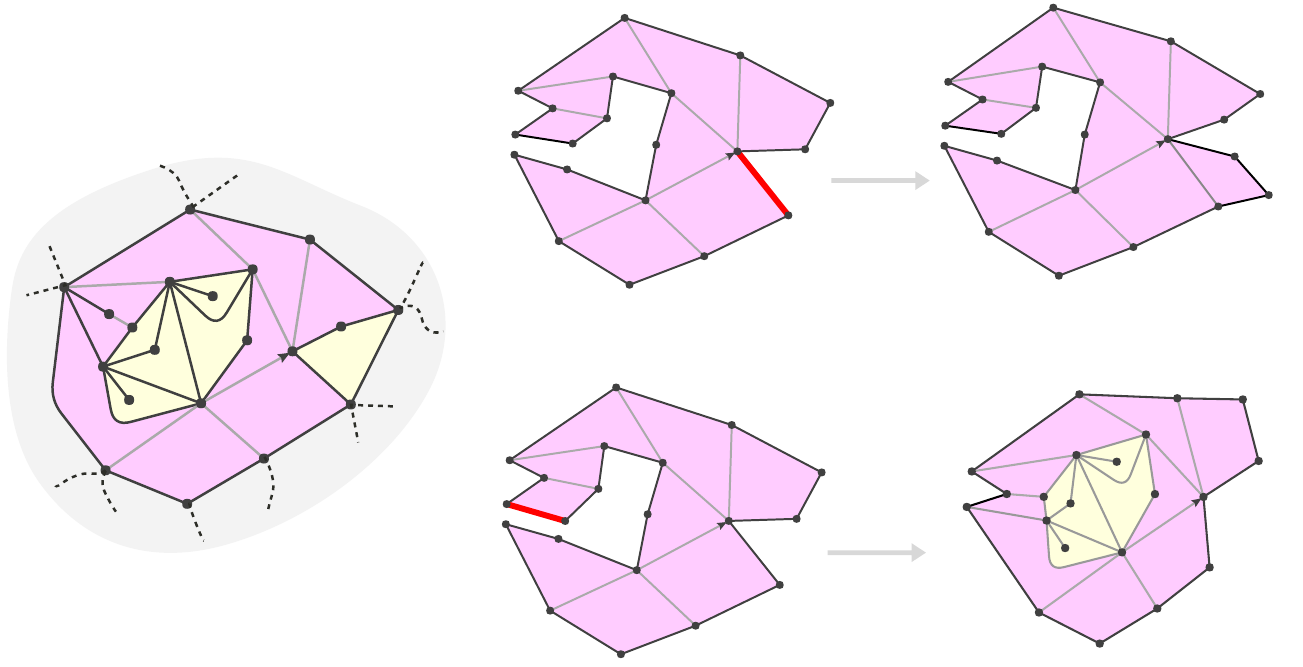}
\caption{Illustration of the filled-in peeling process. In the left-most figure we have explored a certain region $ \e_{n} \subset \Map_{\infty}$ corresponding to the faces in pink glued by the edges in gray. Depending on the edge to peel at the next step we may end-up either with an event of type $ \mathsf{C}_{2}$ (top figures), or an event of type $ \mathsf{G}_{3, *}$ (bottom figures).}
\label{fig:peeling}
\end{figure}
 
During a filled-in peeling exploration $(\e_{n})_{n \geq 0}$ of $\Map_{\infty}$ we denote by $ |\partial \e_{n}|$ the half-perimeter of the boundary of the unique hole of $ \e_{n}$ and by $| \e_{n}|$ the number of inner vertices. The process $ ( | \partial \e_{n}|, | \e_{n}|)_{n \geq 0}$ is a Markov chain whose law is independent of the peeling algorithm with explicit probability transitions~\cite{Budd:The_peeling_process_of_infinite_Boltzmann_planar_maps}. In particular we recall that $(| \partial \e_{n}|)_{n \geq 0}$ is a Doob $h$-transform of a random walk with i.id. increments of law $\nu$ for the function $h^{\uparrow}(p)= 2p\cdot 2^{-2p} \binom{2 p}{p}$ for $p \geq0$ and where $\nu$ satisfies 
\[\nu(-k) \sim  \mathsf{p}_{ \mathbf{q}} \cdot k^{-a}
\qquad\text{and}\qquad
\nu([k, \infty))= \frac{ \mathsf{p}_{ \mathbf{q}}}{a-1}  \cos( a \pi) k^{1-a},\]
where the constant $ \mathsf{p}_{ \mathbf{q}}>0$ depends on our weight sequence. Precise scaling limits for the process $( | \partial \e_{n}|, | \e_{n}|)_{n \geq 0}$ are known  (\cite[Theorem~3.6]{Budd-Curien:Geometry_of_infinite_planar_maps_with_high_degrees} in the case $a\ne 2$\footnote{The case $a=5/2$ is not considered there but the arguments extend readily.} and~\cite[Theorem~1]{Budd-Curien-Marzouk:Infinite_random_planar_maps_related_to_Cauchy_processes} in the case $a= 2$) and in particular it follows that

\begin{prop}[Peeling growth]
\label{prop:volume_peeling}
Let $(\e_n)_{n \ge 0}$ be a filled-in peeling process of $\Map_\infty$. Then we have
\[|\partial \e_n| \approx n^{\frac{1}{a-1}}
\qquad\text{and}\qquad
|\e_n| \approx n^{\frac{a-1/2}{a-1}}.\]
\end{prop}

\subsection{Geometric estimates}
\label{sec:estimees_intermediaires}
We now recall a few geometric estimates that we will use during the proof of our main result. Although some of these estimates may be obtained using the peeling process, we find it more convenient to prove them using Schaeffer-type construction of  $ \Map_{\infty}$~\cite{Bjornberg-Stefansson:Recurrence_of_bipartite_planar_maps,Stephenson:Local_convergence_of_large_critical_multi_type_Galton_Watson_trees_and_applications_to_random_maps}. We postpone the proof of these estimates to Section~\ref{proof:technics}. Recall that $|\map|$ denotes the number of vertices of a map $\map$ and $\rho$ is the origin vertex of $\Map_{\infty}$.

\begin{prop}[Volume growth and tentacles]
\label{prop:volume_boules}
We have
\[|\Ball(\Map_\infty, r)| \approx r^{2a-1},
\quad
|\hBall(\Map_\infty, r)| \approx r^{2a-1},
\quad\text{and}\quad \max\{\dgr(\rho, u); u \in \hBall(\Map_\infty, r)\} \approx r. \]
\end{prop}

This proposition will be deduced from the results  of Le Gall \& Miermont~\cite{Le_Gall-Miermont:Scaling_limits_of_random_planar_maps_with_large_faces} who studied scaling limits of our finite Boltzmann maps conditioned to be large. The next estimates deal with \emph{maps with a boundary}. For $p \geq 1$ we denote $\Map^{(p)}$ a $ \mathbf{q}$-Boltzmann map with a (general) boundary of perimeter $2p$ which is a random bipartite planar map whose law is given by 
\[\mathbb{P}(\Map^{(p)}= \map) \propto \prod_{f  \text{ inner face}} q_{ \mathrm{deg}(f)/2}.\]
This law is important since in a filled-in exploration, during an event of type $ \mathsf{G}_{p, *}$ or $ \mathsf{G}_{*,p}$, the finite hole of perimeter $2p$ is filled-in with a copy of $ \Map^{(p)}$ independently of the past exploration. It can be shown that in the finite map $\Map^{(p)}$ with boundary-length $2p$, the volume grows like $p^{2a-1}$ whilst the diameter grows like $p^{1/2}$. The next result considers the volume growth of the balls in such a map with a boundary.

\begin{prop}\label{prop:volume_boules_cartes_bord}
There exists $c > 0$ and $\lambda > 0$ such that for all $r$ large enough and all $p \ge 2r^2$, we have
\[\Pr{|\Ball(\Map^{(p)}, r)| > \lambda r^{2a-1}} > c.\]
\end{prop}

We can also define $\Map^{(p)}_{\infty}$ the infinite $ \mathbf{q}$-Boltzmann map with boundary of perimeter $2p$ as a local limit of finite maps conditioned to be large, see e.g.~\cite{Curien:Peccot}.  If $\map$ is a map with a boundary we define the \emph{aperture} of $\map$ as  \[\aper( \map) = \max\{ \dgr(x,y) : x,y \mbox{ two vertices on the boundary of }\map\}.\]

\begin{prop}[Growth of the aperture]
\label{prop:aperture}
We have
\[\aper(\Map^{(p)}_{\infty}) \approx p^{1/2}.\]
\end{prop}

\subsection{Proof of the main result}
\label{sec:thm}

We may now prove Theorem~\ref{thm:volume_peeling} relying on the preceding estimates and on properties of the peeling process. Fix any peeling process $(\e_n)_{n \ge 0}$ of $\Map_\infty$, and for every $n \ge 0$, let us denote by $D_n^-$ and $D_n^+$ the smallest and the largest distance \emph{in the whole map} to the origin $\rho$ of a vertex on the boundary $\partial \e_n$. We stress that $D_n^-$ is measurable with respect to $\e_n$, and it equals the smallest distance \emph{in the submap} $\e_n$ to the origin $\rho$ of a vertex on $\partial \e_n$, whereas $D_n^+$ is not measurable with respect to $\e_n$ and is smaller than or equal to the largest distance  in the submap $\e_n$ to the origin $\rho$ of a vertex on $\partial \e_n$.
Clearly, 
\begin{equation}\label{eq:bornes_faciles_peeling}
\hBall(\Map_\infty, D_n^- - 1) \subset \e_n \subset \hBall(\Map_\infty, D_n^+ + 1).
\end{equation}
Theorem~\ref{thm:volume_peeling} thus follows if we prove that
\[D_n^- \gtrsim n^{ \frac{1}{2(a-1)}}
\qquad\text{and}\qquad
D_n^+ \lesssim n^{ \frac{1}{2(a-1)}}.\]
We shall prove these two bounds separately. The second bound $D_n^+ \lesssim n^{ \frac{1}{2(a-1)}}$ is the easy one, it will follow from the same proof technique as that of~\cite{Benjamini-Curien:Simple_random_walk_on_the_uniform_infinite_planar_quadrangulation_subdiffusivity_via_pioneer_points} also recalled in~\cite[Proposition~3.1]{Curien-Marzouk:How_fast_planar_maps_get_swallowed_by_a_peeling_process} and is primarily based on the aperture estimate of Proposition~\ref{prop:aperture}. The first bound $D_n^- \gtrsim n^{ \frac{1}{2(a-1)}}$ will follow from more precise volume consideration. The main idea being that if $D_{n}^{-}$ is small, then a lot of ``large'' peeling steps will accumulate too much volume near the origin of $\Map_{\infty}$.

\paragraph{Upper bound $D_n^+ \lesssim n^{ \frac{1}{2(a-1)}}$ via aperture.} We follow the same lines as~\cite[Proposition 3.1]{Curien-Marzouk:How_fast_planar_maps_get_swallowed_by_a_peeling_process} in our more general context of infinite ``discrete stable'' maps of type $a \in (\frac{3}{2}, \frac{5}{2}]$. Since $|\e_n| \approx n^{\frac{a - 1/2}{a-1}}$ by Proposition~\ref{prop:volume_peeling} and $|\hBall(\Map_\infty, r)| \approx r^{2a-1}$ by Proposition~\ref{prop:volume_boules}, we deduce from~\eqref{eq:bornes_faciles_peeling} the first bounds
\[D_n^- \lesssim n^{\frac{1}{2(a-1)}}
\qquad\text{and}\qquad
D_n^+ \gtrsim n^{\frac{1}{2(a-1)}}.\]
Notice also the easy bound $D_n^+ - D_n^- \le \aper(\Map_\infty \setminus \e_n)$; now recall that the \emph{spatial Markov property} of the peeling exploration  asserts that conditionally on $|\partial \e_n|$, the unexplored region $\Map_\infty \setminus \e_n$ is independent of $\e_n$ and has the law of $\Map_\infty^{(p)}$ with $p = |\partial \e_n|/2$. Since $|\partial \e_n| \approx n^{\frac{1}{a-1}}$ by Proposition~\ref{prop:volume_peeling},  we conclude using Proposition~\ref{prop:aperture}  that
\[D_n^+ - D_n^- \lesssim |\partial \e_n|^{1/2} \approx n^{\frac{1}{2(a-1)}}.\]
Combined with the previous bound, we get $D_n^+ \approx n^{\frac{1}{2(a-1)}}$ as desired.

\paragraph{Lower bound $D_n^- \gtrsim n^{ \frac{1}{2(a-1)}}$ via accumulation of volume near the origin.} As announced, the lower bound will follow from volume consideration. More precisely we shall consider $|\Ball(\e_n, 2r)|$ the number of vertices in the ball of radius $r$ in the submap $ \e_{n}$ centred at the origin $\rho$ of the map $ \Map_{\infty}$ and study its variation $\Delta |\Ball(\e_n, 2r)| = |\Ball(\e_{n+1}, 2r)|-|\Ball(\e_{n}, 2r)|$. Below we write $ (\mathcal{F}_{n})_{n\geq 0}$ for the filtration generated by the peeling process and recall that $ D_{n}^{-}$ as well as $|\partial \e_{n}|$ are measurable with respect to $ \mathcal{F}_{n}$.

\begin{lem}\label{lem:petite_masse}
There exists two constants $K, \lambda > 0$ such that for all $r$ and $n$ large enough, we have
\[\Prc{\Delta |\Ball(\e_n, 2r)| > \lambda r^{2a-1}}{\mathcal{F}_n, D_n^- \le r}
\ge K \cdot |\partial \e_n|^{-(a-1)} \mathbf{1}_{|\partial \e_n| \ge 5r^{2}}.\]
\end{lem}

\begin{proof}
Let us condition on $\mathcal{F}_n$ and on the events $D_n^- \le r$ and $ |\partial  \e_{n}| \geq 4 r^{2}$. Suppose that in the next peeling step, we identify the peel edge with another one on the boundary, separating from infinity a part of the boundary containing (twice) $-\Delta |\partial \e_n| - 1 \ge |\partial \e_n| / 2$ edges. On such an event, by symmetry, there is a chance at least $1/2$ that the boundary swallowed in the finite part contains a vertex $x_n^-$ say, at distance $D_n^- \le r$ from the origin. Then we fill-in this hole by inserting an independent finite Boltzmann map with half-perimeter $-\Delta |\partial \e_n|-1$. Since such a map is invariant under re-rooting along the boundary, we may assume that its root-vertex is matched with $x_n^-$. See Figure~\ref{fig:preuve}.

\begin{figure}[!ht]\centering
\includegraphics[width=.9\linewidth]{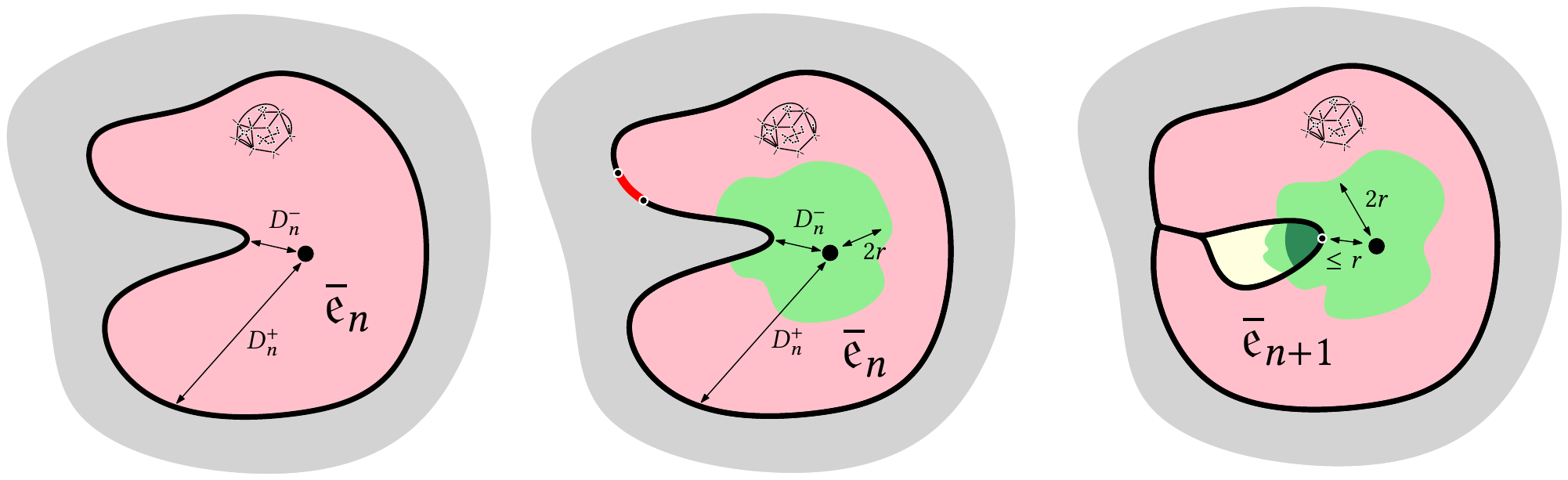}
\caption{Illustration of the proof of Lemma~\ref{lem:petite_masse}. The edge to peel is the red edge. The green region is $\Ball(\e_n, 2r)$. On the event where $- \Delta | \partial \e_{n}| \ge | \partial \e_{n}|/2$ there is a chance at least $1/2$ that the swallowed part of the boundary contains a point at minimal distance (inside $ \e_{n}$) from the origin. If $D_{n}^{-}<r$ then  we add (at least!) to $\Ball( \e_{n+1},2r)$ the ball of radius $r$ in the map filling-in the hole centered at this point  (in dark green on the figure above).}
\label{fig:preuve}
\end{figure}

In this scenario $\Delta |\Ball(\e_n, 2r)|$ is larger than or equal to the volume of the ball of radius $r$ in the map with half-boundary $-\Delta |\partial \e_n|-1$ we just added. According to Proposition~\ref{prop:volume_boules_cartes_bord}, there exist $c > 0$ and $\lambda > 0$ such that for any $p \ge 2r^{2}$, we have $\P(|\Ball(\Map^{(p)}, r)| > \lambda r^{2a-1}) > c$ for all $r$ large enough.
Therefore, for $r$ large, as soon as $| \partial \e_{n}| \ge 4r^2$, on the event $-\Delta |\partial \e_n| - 1 \ge |\partial \e_n| / 2  \ge 2 r^2$, the ball of radius $r$ in the map we add to fill-in the hole has volume at least $\lambda r^{2a-1}$ with probability at least $c$. By the exact transition probabilities of the Markov chain $ | \partial \e_{n}|$ and the facts that $\nu(-k)\sim \mathsf{p}_{ \mathbf{q}}k^{-a}$ and $h^{\uparrow}(p) \sim c' \sqrt{p}$ for some $c' > 0$,
the probability that such a peeling step occurs is bounded below by
\[\Prc{\Delta | \partial \e_{n}| \le - | \partial \e_{n}| / 2}{\mathcal{F}_n}
\geq  \sum_{k = | \partial \e_{n}|/2}^{3| \partial \e_{n}|/4} \frac{h^\uparrow(| \partial \e_{n}|-k)}{h^\uparrow(| \partial \e_{n}|)} \nu(-k)
\ge C \cdot | \partial \e_{n}|^{1-a}.\]
as $| \partial \e_{n}| \to \infty$, for some constant $C>0$. Moreover, given that $\Delta | \partial \e_{n}| - 1 \le - | \partial \e_{n}| / 2$, the probability that a given vertex $x_n^-$ at distance $D_n^- \le r$ from the origin sits on the part of the boundary separated from infinity is at least $1/2$. Gathering-up the pieces we deduce as desired that
\[\Prc{\Delta |\Ball(\e_n, 2r)| > \lambda r^{2a-1}}{\mathcal{F}_n, D_n^- \le r, |\partial \e_n| \ge 4 r^{2}}
\ge  \frac{1}{2} \times c \times  C \cdot |\partial \e_n|^{-(a-1)}\]
for $r,n$ large enough.
\end{proof}

Let us come back to the proof of the lower bound $D_n^- \gtrsim n^{ \frac{1}{2(a-1)}}$ in order to complete the proof of Theorem~\ref{thm:volume_peeling}. Let us fix $\varepsilon > 0$, we aim at showing that there exists $\delta > 0$ such that for all $n$ large enough, we have $\P(D_n^- \le \delta n^{\frac{1}{2(a-1)}}) \le \varepsilon$.  Fix $A$ large (the value of $A$ will be fixed in a few lines by $\varepsilon$), we will first choose  $\eta$ small enough so that 
\begin{equation}\label{eq:eventlot}
\Pr{\sum_{i = 1}^n \frac{K}{|\partial \e_i|^{a-1}} \mathbf{1}_{|\partial \e_n| \ge \eta n^{1/(a-1)}} > A} \ge 1-\varepsilon/4,
\end{equation}
where $K$ is the constant appearing in Lemma~\ref{lem:petite_masse}. This is indeed possible since  by~\cite[Theorem~3.6]{Budd-Curien:Geometry_of_infinite_planar_maps_with_high_degrees} in the case $a\ne 2$ and~\cite[Theorem~1]{Budd-Curien-Marzouk:Infinite_random_planar_maps_related_to_Cauchy_processes} in the case $a= 2$: for each $\eta>0$ the sum in the above display converges in distribution towards 
\[\int_{0}^{1} \frac{ \mathrm{d}t}{(\Upsilon_{t}^{\uparrow})^{a-1}} \mathbf{1}_{\Upsilon_{t}^{\uparrow} > \eta},\]
where $\Upsilon_{t}^{\uparrow}$ is an $(a-1)$-stable L\'evy process conditioned to stay positive (the details of this process can be found in the above references) for which we have  $ \int_{0}^{1} (\Upsilon_{t}^{\uparrow})^{1-a} \mathrm{d} t = \infty$ almost surely by an application of Jeulin's lemma, see~\cite[Corollary~27]{Curien:Peccot}. We now apply Lemma~\ref{lem:petite_masse} with $r \equiv r_{n} = \delta n^{1/(2a-2)}$ with $\delta$ chosen small enough so that $\eta \ge 4\delta^2$: for all $n$ large enough it holds that
\[\Prc{\Delta |\Ball(\e_i, 2 r_n)| > \lambda r_n^{2a-1}}{\mathcal{F}_n, D_i^- \le r_n}
\ge K \cdot |\partial \e_i|^{-(a-1)} \mathbf{1}_{|\partial \e_i| \ge \eta n^{\frac{1}{a-1}}},\]
for all $i \in \{1, \dots, n\}$.  Let us denote by $X_{n,i}$ the indicator of the event in the conditional probability above. Notice that since $D_{i}^{-}$ is non-decreasing, if $D_{n}^{-} \leq r_n$ then $D_{i}^{-} \leq r_n$ for all $1 \leq i \leq n$. By this remark, conditionally on $D_{n} \leq r_{n}$ and on  the event studied in~\eqref{eq:eventlot}, the variable $ \sum_{i=1}^{n} X_{n,i}$ is stochastically lower bounded by a sum of independent Bernoulli random variables $Z_{i}$ with success parameter $ 0<p_{i}<  \eta^{a-1}/i$ and so that $\sum_{i=1}^{n} p_{i} \geq A$ for $n$ large enough. An easy Chernoff bound then shows that
\[\Pr{\sum_{i=1}^{n} Z_{i} >  \frac{A}{8}} \geq 1- \varepsilon/4.\]
When this scenario occurs, the ball of radius $2 r_n$ in $\Map_{\infty}$ contains a volume of at least $A/8$ times $ \lambda r_n^{2a-1}$ whence we deduce that 
\[\Pr{|\Ball( \Map_{\infty}, 2r_n)| \geq  \frac{A}{8} \lambda r_n^{2a-1}} \geq \Pr{D_{n}^{-} \leq  \delta n^{1/(2(a-1))}} - \frac{ \varepsilon}{4} -  \frac{ \varepsilon}{4}.\]
Now, one can further assume that $A$ was chosen large enough so that $ \P(|\Ball( \Map_{\infty}, 2r)| \geq \frac{A}{8} \lambda r^{2a-1}) \leq \frac{\varepsilon}{2}$ for all $r$ large enough by Proposition~\ref{prop:volume_boules}. This finally proves that $\P(D_{n}^{-} \leq  \delta n^{1/(2(a-1))}) \leq  \varepsilon$ as desired.
\qed

\section{Applications}
\label{sec:corollaires}

Let us apply Theorem~\ref{thm:volume_peeling} to three peeling procedures especially designed to study the volume growth of the dual map, the  behaviour of a simple random walk on $\Map_\infty$, and the behaviour of a simple random walk on $\Map_\infty^{\dagger}$.

\subsection{Comparison with the dual map}

\begin{proof}[Proof of Corollary~\ref{cor:volume_primal_dual}]
We apply Theorem~\ref{thm:volume_peeling} to the peeling $\e_n$ with the algorithm $ \mathcal{L}_{ \mathrm{dual}}$ defined in~\cite[Section~2.3]{Budd-Curien:Geometry_of_infinite_planar_maps_with_high_degrees}. Very briefly, in this algorithm, we start with the root-face of $\Map_\infty$ (the one to the right of its root-edge) and we first peel all the edges of this face to reveal the hull of the dual ball of radius $1$. Then we iteratively peel all the edges which, \emph{at this moment} are on the boundary of the explored region to reveal the hull of the dual ball of radius $2$, etc. Note that at every step $n$, the faces incident to the boundary are either at a dual distance $H_n$ say, to the root-face, or at dual distance $H_n+1$. We deduce from Theorem~\ref{thm:volume_peeling} that for every $ \varepsilon>0$, there exist $0 < c_{ \varepsilon}< C_{ \varepsilon} < \infty$ such that for every $n$ large enough, we have
\begin{equation}\label{eq:inclusiondualprimal}
\hBall\left(\Map_\infty, c_{ \varepsilon} n^{\frac{1}{2(a-1)}}\right)
\subset \hBall\left(\Map_\infty^\dag, H_{n} \right)
\subset \hBall\left(\Map_\infty, C_{ \varepsilon} n^{\frac{1}{2(a-1)}}\right),
\end{equation}
with probability at least $1 - \varepsilon$. 
Now, depending on the value $ a \in (3/2;5/2]$ we know the asymptotic behaviour of $H_n$: for $a \in (2;5/2]$, by~\cite[Theorem~4.2]{Budd-Curien:Geometry_of_infinite_planar_maps_with_high_degrees} we have $H_{n} \approx n^{\frac{a-2}{a-1}}$, for $a \in (3/2;2)$, combining~\cite[Theorem~5.3]{Budd-Curien:Geometry_of_infinite_planar_maps_with_high_degrees} and~\cite[Lemma~5.8]{Budd-Curien:Geometry_of_infinite_planar_maps_with_high_degrees}, the ratio $H_n / \log n$ converges in probability to some constant $C_a > 0$, and finally for $a=2$, according to~\cite[Proposition~4]{Budd-Curien-Marzouk:Infinite_random_planar_maps_related_to_Cauchy_processes}, the ratio $H_n / \log^{2} n$ converges in probability to $(2 \pi^{2})^{-1}$. For some $0< c'_{ \varepsilon}<C'_{ \varepsilon} < \infty$ we thus have with probability at least $1- \varepsilon$ when $n$ is large enough,
\[\begin{array}{rccclcl}
c'_{  \varepsilon} n^{ \frac{a-2}{a-1}} &\leq&  H_{n} &\leq& C'_{  \varepsilon} n^{ \frac{a-2}{a-1}}&  \mbox{when}& a \in (2; 5/2],
\\
(1- \varepsilon) C_a \log n &\leq&  H_{n} &\leq&(1+ \varepsilon) C_a \log n & \mbox{when} &a \in (3/2;2),
\\
\frac{1- \varepsilon}{2 \pi^2} \log^2 n &\leq&  H_{n} &\leq& \frac{1+ \varepsilon}{2 \pi^2} \log^2 n & \mbox{when}& a =2.
\end{array}\]
Corollary~\ref{cor:volume_primal_dual} then follows by combining these bounds with~\eqref{eq:inclusiondualprimal} and using monotonicity properties. \end{proof}

\subsection{Pioneer points and sub-diffusivity}\label{sec:pionniers}

\paragraph{Walk on $ \Map_{\infty}$}
Let $X = (X_n)_{n \ge 0}$ be the simple random walk on $\Map_\infty$ started from the origin $\rho$, which should be viewed as a sequence $(\vec{E}_n)_{n \ge 0}$ of oriented edges such that $\vec{E}_0$ is the root-edge, and for every $n \ge 0$, conditional on the edge $\vec{E}_n$, we choose one of the edges incident to the tip of $\vec{E}_n$ uniformly at random, then $\vec{E}_{n+1}$ is this new edge, oriented from the tip of $\vec{E}_n$. Then $X_n$ is the origin of the edge $\vec{E}_n$. We say that $\tau \ge 0$ is a \emph{pioneer time} if $X_\tau$ lies on the boundary of the unique infinite component when we remove all the faces incident to one of the $X_i$'s for $i < \tau$; then $X_\tau$ is called a \emph{pioneer point} (so $X_0 = \rho$ is a pioneer point). For every $n \ge 0$, we let $\mathcal{P}_n$ be the $n$-th pioneer point.

\begin{cor}
\label{cor:points_pionniers_sous_diff}
Fix a critical weight sequence $\q$ of type $a \in (\frac{3}{2}, \frac{5}{2}]$. We have
\[ \sup_{1 \le k \le n} \dgr(\rho, \mathcal{P}_k) \approx  n^{\frac{1}{2(a-1)}}
\qquad\text{and}\qquad
\sup_{0 \le k \le n} \dgr(\rho, X_k) \lesssim n^{\frac{1}{2(a-1)}}.\]
\end{cor}

\begin{proof}[Proof of Corollary~\ref{cor:points_pionniers_sous_diff}]
We use the peeling algorithm $\mathcal{S}_{ \mathrm{primal}}$ defined in~\cite[Section~1.4]{Benjamini-Curien:Simple_random_walk_on_the_uniform_infinite_planar_quadrangulation_subdiffusivity_via_pioneer_points} which follows the walk $X$: if $X_{n}$ is not a pioneer point, it lies in the interior of the submap we have revealed so far and we may directly move to $X_{n+1}$. If otherwise $X_{n}$ is a pioneer point, then it lies on the boundary of the explored part and we peel the edge on the boundary which lies immediately to the left of $X_{n}$. We continue to do so until $X_{n}$ does not belong to the boundary of the explored part: The $1$-neighborhood of $X_{n}$ has then been completely explored and we may  perform a random walk step.

Let us denote by $(\e_n)_{n \ge 0}$ the associated filled-in peeling process and write $\theta_{n}$ for the number of pioneer points we have encountered in the first $n$ peeling steps. Since we only peel when the walk is at a pioneer point we have $\theta_{n} \leq n+1$. On the other hand, if $X_{n}$ is on the boundary of the explored part of perimeter say $2 p$, there is a probability at least $ \nu(-1)h^{\uparrow}(p-1)/h^{\uparrow}(p)$ that $X_{n}$ is swallowed by a peeling step of type $ \mathsf{G}_{*,0}$ and is not exposed on the boundary of the explored part anymore. If $p=1$ the point $X_{n}$ can be swallowed in two peeling steps. Since $\inf_{ p \geq 2}\nu(-1)h^{\uparrow}(p-1)/h^{\uparrow}(p) >c >0$ we see that the time it takes to discover the neighbourhood of a given pioneer point is stochastically dominated by a geometric random variable. It easily follows that
\[\theta_{n} \approx n.\]
Using this and Theorem~\ref{thm:volume_peeling} it follows that the first $n$ pioneers points of the walk, and a fortiori the first $n$ steps of the walk, take place within $ \hBall(\Map_\infty, C_{ \varepsilon} n^{\frac{1}{2(a-1)}})$ with probability at least $1- \varepsilon$. Using  the third item of Proposition~\ref{prop:volume_boules} it follows that 
\[ \sup_{1 \le k \le n} \dgr(\rho, \mathcal{P}_k) \lesssim  n^{\frac{1}{2(a-1)}}
\qquad\text{and}\qquad
\sup_{0 \le k \le n} \dgr(\rho, X_k) \lesssim n^{\frac{1}{2(a-1)}}.\]
For the lower bound $\sup_{1 \le k \le n} \dgr(\rho, \mathcal{P}_k) \gtrsim  n^{\frac{1}{2(a-1)}}$ notice that the $n$-th pioneer point is necessarily outside $ \e_{n-1}$ and so by Theorem~\ref{thm:volume_peeling} it must be at distance at least $ c_{ \varepsilon} n^{ \frac{1}{2(a-1)}}$ from the origin of the map with probability at least $ 1- \varepsilon$.\end{proof}

\paragraph{Walk on $ \Map_{\infty}^{\dagger}$}
We can use the same strategy as in the last section to study the random walk on $\Map_{\infty}^{\dagger}$. More precisely, let us denote by $X^\dag = (X^\dag_n)_{n \ge 0}$ the simple random walk on $\Map_\infty^\dag$ started from the root-face $\rootface$. As before, one can design an algorithm $\mathcal{S}_{ \mathrm{dual}}$ that explores the map along the walk (see~\cite[Section 8.2.2]{Curien:Peccot}). The latter is simpler than $ \mathcal{S}_{ \mathrm{primal}}$: the walk starts from the root-face $\rootface$ and wants to traverse one of the edges of this face, we then peel this edge and reveal the face on the other side before moving to $X^\dag_1$. Then we continue like this: at each step, either the walk wants to traverse an edge such that the other side has not been revealed yet, in which case we peel this edge, or the other side is known and the walk can directly move. We call \emph{pioneer edges} the edges traversed by the walk that lead to the triggering of a peeling step; for all $k \ge 0$, let us denote by $\mathcal{P}^\dag_k$ the $k$-th pioneer edge. Finally, we denote by $\dgr^\dag$ the graph distance in $\Map_\infty^\dag$.

\begin{cor}\label{cor:points_pionniers_sous_diff_dual}
Fix a critical weight sequence $\q$ of type $a \in (\frac{3}{2}, \frac{5}{2}]$; there exists $c_a > 0$ such that the following holds: For every $ \varepsilon>0$, there exist $0 < c_{ \varepsilon} < C_{ \varepsilon}< \infty$ such that for every $r$ large enough, we have 
\[\begin{array}{rccclcl}
c_{  \varepsilon} n^{ \frac{a-2}{a-1}} &\leq&  \displaystyle \max_{1 \le k \le n} \dgr^\dag(\rootface, \mathcal{P}^\dag_k) &\leq& C_{  \varepsilon} n^{ \frac{a-2}{a-1}}&  \mbox{when}& a \in (2, 5/2],
\\
(1- \varepsilon) c_a \log n &\leq&  \displaystyle \max_{1 \le k \le n} \dgr^\dag(\rootface, \mathcal{P}^\dag_k) &\leq&(1+ \varepsilon) c_a \log n & \mbox{when} &a \in (3/2,2),
\\
\frac{1- \varepsilon}{2 \pi^2} \log^2 n &\leq&  \displaystyle \max_{1 \le k \le n} \dgr^\dag(\rootface, \mathcal{P}^\dag_k) &\leq& \frac{1+ \varepsilon}{2 \pi^2} \log^2 n & \mbox{when}& a =2.
\end{array}\]
with probability at least $1 - \varepsilon$.
\end{cor}

\begin{proof}
Let $ (\e_{n})_{n \ge 0}$ be the filled-in peeling process associated with $\mathcal{S}_{ \mathrm{dual}}$. Using the fact that the submaps $(\e_{k})_{0 \leq k \leq n}$ are nested, a moment's though shows that 
\[\min_{f \in \partial \e_n} \dgr^\dag(\rootface, f)
\leq  \dgr^\dag(\rootface, \mathcal{P}^\dag_n)
\leq \boldsymbol{\max_{1 \le k \le n} \dgr^\dag(\rootface, \mathcal{P}^\dag_k)}
\leq \max_{1 \le k \le n} \max_{f \in \partial \e_k} \dgr^\dag(\rootface, f)
= \max_{f \in \partial \e_n} \dgr^\dag(\rootface, f),\]
where by $ f \in \partial \e_{n}$ we mean a face incident to the boundary $\partial \e_n$. By Theorem~\ref{thm:volume_peeling}, the smallest and the largest \emph{primal} graph distance to the root-vertex of the boundary $\partial \e_{n}$ are both of order $n^{1/(2a-2)}$, we then conclude from Corollary~\ref{cor:volume_primal_dual} that $\min_{f \in \partial \e_n} \dgr^\dag(\rootface, f)$ and $\max_{f \in \partial \e_n} \dgr^\dag(\rootface, f)$ satisfy respectively the lower and upper bounds of our claim.
\end{proof}

We point out that, as opposed to Corollary~\ref{cor:points_pionniers_sous_diff}, this result does not imply upper bounds for the quantities $\max_{1 \le k \le n} \dgr^\dag(\rootface, X^\dag_k)$ because we do not have the last claim of Proposition~\ref{prop:volume_boules} for the dual map: We do get that the walk $X^\dag$ up to time $n$ stays within a hull $\hBall(\Map_\infty^\dag, r_n)$ for some $r_n$ given by Corollary~\ref{cor:points_pionniers_sous_diff_dual}, but this hull may have ``tentacles'' reaching distance much larger than $r_{n}$ (at least in the dense regime, but it should not be the case in the dilute regime).

\section{Maps as labelled trees and geometric estimates}
\label{sec:arbres_cartes}

In this section, we recall the other very efficient tool to study planar maps which is a construction from \emph{labelled mobiles} originally due to Bouttier, Di Francesco, and Guitter; let us first define these objects before recalling the construction (we refer the reader to~\cite[Section~6]{Bettinelli-Miermont:Compact_Brownian_surfaces_I_Brownian_disks} for details).

\subsection{Bouttier--Di Francesco--Guitter coding of bipartite maps}
\label{sec:BDG}

\paragraph{Finite maps.}
Let us set $\Z_{\ge -1} = \{-1, 0, 1, 2, \dots\}$ and for every $k \ge 1$, consider the following set of bridges:
\[\mathcal{B}_k^{\ge -1} = \left\{(b_0, \dots, b_k): b_0=b_k=0 \text{ and } b_j-b_{j-1} \in \Z_{\ge -1} \text{ for } 1 \le j \le k\right\}.\]
A \emph{mobile} is a finite rooted plane tree whose vertices at even (resp. odd) generations are white (resp. black). We consider \emph{labelled} mobiles, in which every white vertex $u$ carries a label $\ell(u) \in \Z$. We say that a finite ordered forest of mobiles $(\mathfrak{t}_1,\dots,\mathfrak{t}_p)$ is \emph{well-labelled} if
\begin{enumerate}
\item The sequence of labels of the roots of $\mathfrak{t}_1,\dots,\mathfrak{t}_p, \mathfrak{t}_1$ belongs to $\mathcal{B}_{p+1}^{\ge -1}$;
\item For every black vertex $u$, if $u_0$ denotes its white parent and $u_1, \dots, u_k$ its white children, ordered from left to right, then the sequence of labels $(\ell(u_0), \ell(u_1), \dots, \ell(u_k), \ell(u_0))$ belongs to $\mathcal{B}_{k+1}^{\ge -1}$. 
\end{enumerate}

Imagine that the forest $\mathfrak{t}_1, \dots, \mathfrak{t}_p$ is properly drawn on the plane inside a cycle of length $p$, with the roots grafted in counter-clockwise order on the cycle. Let us define the \emph{white contour sequence} $(c^\circ_n)_{n \ge 0}$ as the sequence of corners formed by the white vertices, starting from $c^\circ_0$ the corner between the root of $\mathfrak{t}_1$ and its first black child (if any, otherwise the corner formed by this root only), and following the contour of the forest from left to right cyclically. Recall that the white vertices are labelled, we associate with each white corner the label of the corresponding vertex; we then say that a corner $c^\circ_j$ is the \emph{successor} of another corner $c^\circ_i$ if $c^{\circ}_{j}$ is the first corner in the cyclic contour after $c^{\circ}_{i}$ such that $\ell(c^\circ_j) = \ell(c^\circ_i) - 1$. For this definition to hold also when $\ell(c^\circ_i) = \ell_{\min}$ is the overall minimum of labels, we add an extra vertex $u_\star$ labelled $\ell_{\min}-1$. 

We associate a \emph{pointed} planar map ---~i.e. a map with a distinguished vertex~--- with such a well-labelled forest of mobiles by drawing the links between each corner and its successor in a non-crossing fashion and then erasing the embedding of the cycle and the edges of the mobiles; we are then left with a bipartite map on the set of white vertices of the forest and the distinguished vertex $u_\star$, with a black vertex inside each inner face, and the degree of this vertex in its mobile is half the degree of the face in the map. The external face of the map is the face that ``encloses'' the cycle on which the mobiles have been grafted. The root-edge of the map is not prescribed by the forest and is taken uniformly at random on the external face of degree $2p$ (oriented so that the external face is on its right). 
The labelling of the above forest has a geometric interpretation in terms of the map: the label of a vertex minus $\ell_{\min}$ plus one is the graph distance in the map to the distinguished vertex $u_\star$. 
See Figure~\ref{fig:bdg} for an illustration.

\begin{figure}[!ht]\centering
\includegraphics[width=.85\linewidth]{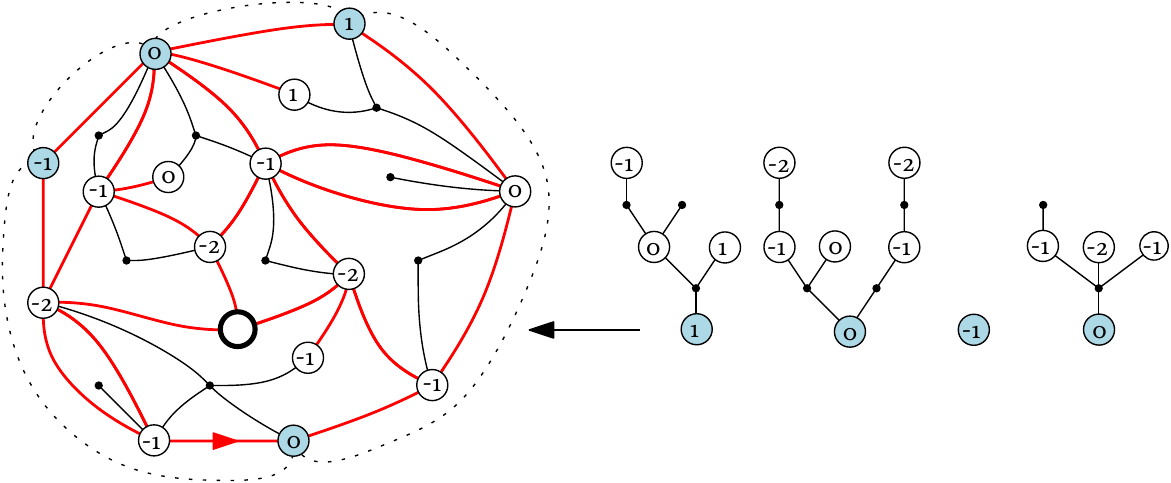}
\caption{Illustration of the construction of a pointed (at the white unlabelled vertex) bipartite planar map with a boundary of perimeter $8$ (in red) from a forest of $4$ mobiles. Note that the boundary is simple here, which may not be the case in general.}
\label{fig:bdg}
\end{figure}

\paragraph{Infinite maps.}
Let us next briefly extend the preceding construction to infinite maps with boundary-length $2p$. We start as above with a forest of mobiles $(\mathfrak{t}_1,\dots,\mathfrak{t}_p)$, where the $\mathfrak{t}_i$'s for $i \ge 2$ are as above, but $\mathfrak{t}_1$ is now an infinite mobile with one end, in the sense that it is locally finite and there is a unique self-avoiding infinite path, called thereafter the \emph{spine}, so the tree can be constructed from such a half-line of alternating white and black vertices $s^\circ_0, s^\bullet_0, s^\circ_1, s^\bullet_1, s^\circ_2, s^\bullet_2, \dots$ on which are grafted finite trees. This spine splits the forest into two parts: the one on its left made of all the trees grafted to the left of the spine, and the one on the right made of all the trees grafted to the right of the spine and the other $\mathfrak{t}_i$'s. Then we may define a white contour sequence as a process indexed by $\Z$: define $(c^\circ_n)_{n \ge 0}$ following the contour of $\mathfrak{t}_1$ starting as above from $c^\circ_0$ the corner between the root of $\mathfrak{t}_1$ and its first black child; on the other hand define $(c^\circ_{-n})_{n \ge 1}$ following the contour of the forest but now from right to left: $c^\circ_{-1}$ is the corner between the root of $\mathfrak{t}_p$ and its last black child and then we visit all the white corners of the $\mathfrak{t}_i$'s for $i \ge 2$ before reaching $\mathfrak{t}_1$ and following the part to the right of the spine.

As before we consider labels on the white vertices and we say that the forest is well-labelled when the labels satisfy the same local rule as in the finite case and furthermore the set of labels on the spine is not bounded below. We then construct a map as previously, by first imagining that the trees are properly drawn inside a cycle of length $p$ with a unique accumulation point (corresponding to the infinite tree) and then linking every white corner to its successor in a non-crossing fashion. Our assumption ensures that this is always possible, so there is no need to add any extra vertex here (the distinguished vertex is ``sent to infinity''). The root-edge is chosen uniformly at random on the external face as above. We refer to~\cite[Section~6.3]{Stephenson:Local_convergence_of_large_critical_multi_type_Galton_Watson_trees_and_applications_to_random_maps} and~\cite[Section 2]{Bjornberg-Stefansson:Recurrence_of_bipartite_planar_maps}  for this construction in the case $p=1$.

\subsection{The distribution of random labelled mobiles}
\label{sec:GW_bi_types}

A way to generate random  bipartite pointed Boltzmann planar maps consists in constructing it as previously, starting from a random forest.  Let $\T$ be an \emph{alternating two-type Bienaym\'{e}--Galton--Watson tree}: it has white and black vertices at even and odd generation respectively, which reproduce independently according to the following offspring distributions:
\[\mu_\circ(k) = Z_\q^{-1} (1 - Z_\q^{-1})^k
\qquad\text{and}\qquad
\mu_\bullet(k) = \frac{Z_\q^{k+1} \binom{2k+1}{k} q_{k+1}}{Z_q-1}\]
for all $k \ge 0$. Recall the law $\mu$ from the introduction, note that for $k \ge 1$, the ratio $\mu(k) / \mu_\bullet(k-1)$ is constant, so $\mu_\bullet$ or $\mu$ has finite variance or belongs to the strict domain of attraction of an $(a-1)$-stable distribution if and only if the other satisfies the same property. Furthermore, easy calculations show that $\mu$ has mean one if and only if the product of the means of $\mu_\circ$ and $\mu_\bullet$ equals one, so the two-type tree is critical. A simple and useful observation is that the tree induced by the white vertices, given by keeping only these white vertices and linking each one to its white grand-parent, is a Bienaym\'{e}--Galton--Watson forest; we shall denote the offspring distribution by $\tilde{\mu}$, which slightly differs from $\mu$ but has the same ``stable behaviour'', see~\cite[Section~3.2]{Le_Gall-Miermont:Scaling_limits_of_random_planar_maps_with_large_faces}.

Let $\T_1, \dots, \T_p$ be i.id. copies of $\T$ and, conditionally on this forest, sample  labels uniformly at random amongst all possibilities which make the forest well-labelled; this just means that at every black vertex with, say, $k-1$ offsprings, the sequence of labels around it in clockwise order forms a uniformly random bridge in $\mathcal{B}_k^{\ge -1}$ shifted by the value of the label of its parent, independently of the rest, and similarly for the roots. The law of that bridge is the same as that of a random walk bridge of length $k$, with i.id. increments of law
\begin{equation}\label{eq:increments_bridge}
\Pr{\xi=k} = 2^{-2-k} \qquad\text{for } k \geq -1.
\end{equation}
Then~\cite[Proposition 22]{Bettinelli-Miermont:Compact_Brownian_surfaces_I_Brownian_disks} shows that the pointed map constructed as above from $\T_1, \dots, \T_p$ has the law of a $\q$-Boltzmann \emph{pointed} planar map with a boundary with length $2p$ which we denote by 
$ \Map^{(p)}_{\bullet}$, i.e. $ \P( \Map^{(p)}_{\bullet} = (\map, \bullet)) \propto w( \map)$. This is not quite the desired law $\Map^{(p)}$ but the latter can be obtained by a simple bias: for every non-negative function $f$ that depends on the map only (not of its pointing),
\begin{equation}\label{eq:sizebias}
\Es{f( \Map^{(p)})} =  \frac{1}{ \E[1/ | \Map^{(p)}_{\bullet}|]} \cdot \Es{f ( \Map^{(p)}_{\bullet}) \cdot \frac{1}{| \Map^{(p)}_{\bullet}|}},
\end{equation}
where $|\map|$ is the number of vertices of the map $\map$.  
 
 Similarly, if the random labelled forest $(\T_1, \dots, \T_p)$ has the same law as above, except that $\T_1$ is infinite  and has the law of $\Tree_\infty$, the two-type Bienaym\'{e}--Galton--Watson trees as above \emph{conditioned to survive} then the associated map has the law $\Map^{(p)}_{\infty}$ (this follows from~\cite[Proposition 22]{Bettinelli-Miermont:Compact_Brownian_surfaces_I_Brownian_disks} and the work~\cite{Stephenson:Local_convergence_of_large_critical_multi_type_Galton_Watson_trees_and_applications_to_random_maps}). The law of $\Tree_\infty$  may be constructed in the following way: all the vertices reproduce independently, the ones outside the spine reproduce according to their respective offspring distribution, and the ones on the spine $s^\circ_i, s^\bullet_i$ reproduce according the \emph{size-biased} versions of these laws; finally, the offspring of a vertex on the spine which belongs to the spine is chosen uniformly at random.

\subsection{Asymptotic estimates on labelled mobiles}
\label{sec:estimees_arbres_etiquetes}

Let us consider a sequence $(\T_n)_{n \ge 1}$ of i.id. well-labelled mobiles with the same distribution as $\T$, that we view as an ordered forest; the labels of the roots of the mobiles is zero and the rest of the labels of the mobiles is sampled uniformly at random as above. Let $S^\circ = (S^\circ_k)_{k \ge 0}$ and $L^\circ = (L^\circ_k)_{k \ge 0}$ be respectively the white \emph{{\L}ukasiewicz walk} and the \emph{label process} associated with this forest, constructed as follows: let us read the white vertices of the forest in depth-first search order, starting at $0$ from the root of the first tree, then put $S^\circ_0 = 0$ and for every $k \ge 0$, let the difference $S^\circ_{k+1} - S^\circ_k$ record the number of grand-children minus one of the $k$-th white vertex (so $S^\circ$ is nothing but a centred random walk with step distribution $\tilde{\mu}(\cdot + 1)$), and let $L^\circ_k$ denote the label of this $k$-th white vertex. According to Le Gall \& Miermont~\cite[Theorem~1]{Le_Gall-Miermont:Scaling_limits_of_random_planar_maps_with_large_faces}, for $a \in (\frac{3}{2}, \frac{5}{2})$, we have the convergence in distribution in the Skorohod space
\begin{equation}\label{eq:convergence_LGM}
\left(n^{-\frac{1}{a-1/2}} S^\circ_{\lfloor nt\rfloor}, n^{-\frac{1}{2a-1}} L^\circ_{\lfloor nt\rfloor}\right)_{t \ge 0}
\cvdist
(c_0 \mathcal{S}_t, \sqrt{2c_0} \mathcal{Z}_t)_{t \ge 0},
\end{equation}
where $c_0$ is some constant depending on $\tilde{\mu}$, where $\mathcal{S}$ is an $(a-1/2)$-stable L\'{e}vy process with no negative jump, and the process $\mathcal{Z}$ is the \emph{continuous distance process} constructed in~\cite{Le_Gall-Miermont:Scaling_limits_of_random_planar_maps_with_large_faces}. 

Let us say a few words about this process $\mathcal{Z}$. In the discrete setting, the label of a white vertex is the sum of the label increments along its ancestral line, between each white ancestor, $u$ say, and its grand-parent, $v$ say, and these increments are given by the value $B(k,j)$ of an independent uniformly random bridge with jumps in $\Z_{\ge -1}$ of length $k$ at time $j$, where $k$ is the degree of the black vertex between $u$ and $v$, and $j$ is the position of $u$ amongst its siblings. The ancestors of the $n$-th white vertex are given by those times $m \le n$ such that $S^\circ_m \le \min_{[m+1,n]} S^\circ$, and the values $k$ and $k-j$ associated with this ancestor are encoded in the {\L}ukasiewicz path: suppose for simplicity that $v$ has only one black child, then $k$ and $k-j$ are given respectively by $S^\circ_{m+1} - S^\circ_m+1$ and $\min_{[m+1,n]} S^\circ - S^\circ_m$. At the continuum level, the construction of $\mathcal{Z}$ is similar: conditional on $\mathcal{S}$, for every $t > 0$, the value of $\mathcal{Z}_t$ is given by the sum of independent Brownian bridges of length given by the jumps $\mathcal{S}_s - \mathcal{S}_{s-}$ and evaluated at times given by $\inf_{[s,t]} \mathcal{S} - \mathcal{S}_{s-}$, only for those times $s < t$ such that $\inf_{[s,t]} \mathcal{S} > \mathcal{S}_{s-}$. It is shown in~\cite{Le_Gall-Miermont:Scaling_limits_of_random_planar_maps_with_large_faces} that such a process is well-defined and admits a continuous version.

In the case $a = \frac{5}{2}$, the convergence~\eqref{eq:convergence_LGM} still holds, where $\mathcal{S}$ is now a Brownian motion and $\mathcal{Z}$ is the so-called \emph{head of the Brownian snake} driven by $ \mathcal{S}$, which can be viewed as a Brownian motion indexed by the Brownian forest encoded by $\mathcal{S}$; the argument may be adapted from~\cite{Marzouk:On_scaling_limits_of_planar_maps_with_stable_face_degrees} which considers size-conditioned trees with offspring distribution $\mu$ instead.

We next derive a version of~\eqref{eq:convergence_LGM} for the tree $\Tree_\infty$ conditioned to survive. Let us define similarly its {\L}ukasiewicz path $S^\infty$ and its label process $L^\infty$ by restricting to the white vertices on the left part of the tree. It is known that $S^\infty$ has the law of the random walk $S^\circ$ conditioned to always stay non-negative (see e.g.~\cite{Bertoin-Doney:On_conditioning_a_random_walk_to_stay_nonnegative}), which can be rigorously defined as the Doob $h$-transform using the harmonic function $h(n) = (n+1) \ind{n \ge 0}$. Similarly, the L\'{e}vy process $\mathcal{S}$ can be conditioned to stay positive via such an $h$-transform and we denote by $\mathcal{S}^\uparrow$ this process, see the introduction of~\cite{Caravenna-Chaumont:Invariance_principles_for_random_walks_conditioned_to_stay_positive} and references therein. One can finally adapt the construction of the process $\mathcal{Z}$ from $\mathcal{S}$ in~\cite{Le_Gall-Miermont:Scaling_limits_of_random_planar_maps_with_large_faces} to this setting and define a process $\mathcal{Z}^\uparrow$ from $\mathcal{S}^\uparrow$ when $a < \frac{5}{2}$; when $a = \frac{5}{2}$ the process $\mathcal{Z}^\uparrow$ is simply the head of the Brownian snake driven by a three-dimensional Bessel process (Brownian motion conditioned to stay positive).

\begin{prop}\label{prop:convergence_processus}
We have the convergence in distribution for the Skorokhod topology 
\[\left(n^{-\frac{1}{a-1/2}} S^\infty_{\lfloor nt\rfloor}, n^{-\frac{1}{2a-1}} L^\infty_{\lfloor nt\rfloor}\right)_{t \ge 0}
\cvdist
(c_0 \mathcal{S}^\uparrow_t, \sqrt{2c_0} \mathcal{Z}^\uparrow_t)_{t \ge 0},\]
\end{prop}

\begin{proof}[Proof sketch.]
Since the path $S^\infty$ has the law of the random walk $S^\circ$ conditioned to stay non-negative, the convergence of the former follows from that of the latter in~\eqref{eq:convergence_LGM}, see Caravenna \& Chaumont~\cite{Caravenna-Chaumont:Invariance_principles_for_random_walks_conditioned_to_stay_positive}. 
Let us next focus on the convergence of the finite-dimensional marginals of the label process. By appealing to Skorokhod's representation theorem, we may assume that the convergence of the {\L}ukasiewicz path holds almost surely. Recall the construction of the process $L^\infty$ from random bridges associated with each black branch-point. When $a < \frac{5}{2}$, the proof goes exactly as that of Proposition~7 in~\cite{Le_Gall-Miermont:Scaling_limits_of_random_planar_maps_with_large_faces}, it suffices to only consider the large black branch-points since the contribution of all the others is small; these large branch-points, with  length of order $n^{1/(a-1/2)}$, give at the limit, after a diffusive scaling $n^{1/(2a-1)}$, independent Brownian bridges, and the sum of these bridges evaluated at the corresponding times, along the ancestral line of a point is the definition of $\mathcal{Z}^\uparrow_t$. When $a = \frac{5}{2}$, we may similarly adapt the argument from~\cite{Marzouk:On_scaling_limits_of_planar_maps_with_stable_face_degrees}: here the branch-points are too small and the label increments between a white individual in the tree and its white grand-parent behave almost like i.i.d. random variables with finite variance, which gives at the limit a Brownian motion indexed by the infinite Brownian tree, which again is the definition of $\mathcal{Z}^\uparrow_t$.

Finally, tightness of $L^\infty$ follows by absolute continuity considerations with respect to the infinite forest. Indeed, for any $N \ge 1$ fixed, the law of the pair $(S^\infty_k, L^\infty_k)_{k \le N}$ is absolutely continuous with respect to the similar pair associated with a mobile conditioned to have more than $2N$ white vertices, which has the law of the first mobile in the infinite forest with more than $2N$ white vertices. This is well-known for the {\L}ukasiewicz path and more generally for conditioned random walk, and it extends to the label process by construction, since the latter is obtained from the {\L}ukasiewicz path and independent random bridges.
\end{proof}

For every $r \ge 0$, let $\sigma_r$ denote the first instant $i \ge 0$ such that the $i$-th white vertex of $\Tree_\infty$ is on its spine, and it is the first one on the spine with label smaller than $-r$. The preceding proposition yields the following asymptotic behaviour.

\begin{cor}\label{cor:temps_sortie_spine}
We have
\[\sigma_r \approx r^{2a-1}
\qquad\text{and}\qquad
\max_{k \le \sigma_r} |L^\infty_k| \approx r.\]
\end{cor}

\begin{proof}
It is clear from the definition that the times of visit of a white vertex on the spine correspond to those times $i \ge 0$ such that $S^\infty_i = \min_{j \ge i} S^\infty_j$. The continuum analogue of $\sigma_r$ is the first-passage time $ \Sigma_{\alpha}$  below $-\alpha<0$ of the process $ \sqrt{2 c_{0}}\mathcal{Z}^\uparrow$ restricted to those times $t \ge 0$ such that $\mathcal{S}^\uparrow_s \ge \mathcal{S}^\uparrow_t$ for all $s \ge t$, which is easily seen to be finite for all $\alpha >0$. We claim that 
\[\left(r^{-(2a-1)} \sigma_{r}, r^{-1} \max_{k \le \sigma_r} |L^\infty_k|\right) \cvdist[r] \left(\Sigma_{1}, \sup_{0 \leq s \leq \Sigma_{1}} | \sqrt{2 c_{0}} \mathcal{Z}^{\uparrow}_{s} | \right).\]
Indeed, in the case $a = 5/2$, the label process along the spine behaves as a Brownian motion which, almost surely, takes values smaller than $-1$ immediately after reaching $-1$, so the last display is implied by Proposition~\ref{prop:convergence_processus}. In particular $\sigma_r \approx r^{2a-1}$ and $ \max_{k \le \sigma_r} |L^\infty_k|  \approx r$ when $a=5/2$. When $a < 5/2$, the same phenomenon occurs, and in fact the limiting process of labels on the spine is a $2(a-3/2)$-stable symmetric L\'evy process, which jumps strictly below $-1$ when entering $(-\infty,-1]$. We conclude similarly.
\end{proof}

\subsection{Proof of the geometric estimates on maps}
\label{proof:technics}

We prove in this final section the volume estimates from Section~\ref{sec:estimees_intermediaires} we used in the proof of Theorem~\ref{thm:volume_peeling}, appealing to the results from the preceding section on labelled forests. Let us first start by considering $\Map_\infty$ and proving Proposition~\ref{prop:volume_boules} on the balls and their hulls, that is
\[|\Ball(\Map_\infty, r)| \approx r^{2a-1},
\quad
|\hBall(\Map_\infty, r)| \approx r^{2a-1},
\quad\text{and}\quad
\max\{\dgr(\rho, u); u \in \hBall(\Map_\infty, r)\} \approx r.\]

\begin{proof}[Proof of Proposition~\ref{prop:volume_boules}]
We suppose that $\Map_{\infty}$ is constructed from  $\Tree_\infty$ as in the last section (this corresponds to the case $p=1$) and let us suppose for convenience that we rooted the map is such a way that the origin vertex is the origin of the tree (otherwise it is at distance at most $1$ from it). Recall that $\Tree_\infty$ has a spine, and denote by $  \mathbb{T}_{r}$ the tree $\Tree_{\infty}$ obtained by chopping off the descendant of the first white vertex on the spine whose label drops below $-r-3$. Using the well-known ``cactus bound'', the proof of~\cite[Equation~19]{Benjamini-Curien:Simple_random_walk_on_the_uniform_infinite_planar_quadrangulation_subdiffusivity_via_pioneer_points} shows \emph{mutatis mutandis} (considering only white vertices and corners) that the following inclusion holds in terms of white vertices in $\Map_\infty$:
\begin{equation}\label{eq:inclusion1}
\hBall(\Map_\infty, r) \subset  \mathbb{T}_{r}.
\end{equation}
Using the notation of Corollary~\ref{cor:temps_sortie_spine}, the number of white vertices on the ``spine'' of $ \mathbb{T}_{r}$ and to its left is given by $\sigma_{r+3}$, and the number of white vertices on the spine and to the right has the same law by symmetry. By Corollary~\ref{cor:temps_sortie_spine} the number of white vertices of $ \mathbb{T}_{r}$ is  is therefore of order $r^{2a-1}$; note that we counted twice the $n$ vertices on the spine, which is negligible. 

Newt we claim that 
\begin{equation}\label{eq:tentacle1}
\max\{\dgr(\rho, u); u \in  \mathbb{T}_{r}\}
\le 2 + 3 \max_{u \in  \mathbb{T}_{r}} |\ell(u)|
\approx r
\end{equation}
where the distance $ \dgr$ is in the map $\Map_{\infty}$ and $\rho$ is its origin vertex. Indeed, the chain of successors starting from any white vertex in $ \mathbb{T}_{r}$ must coalesce with the chain starting from the root corner and this produces a path between those two vertices of length bounded above by $2+3\max_{u \in \mathbb{T}_{r}}|\ell(u)|$. This variable is of order $r$ by  Corollary~\ref{cor:temps_sortie_spine}.

We can then prove the three points of the proposition. The third point follows from~\eqref{eq:tentacle1} and~\eqref{eq:inclusion1} after noting that certainly $\max\{\dgr(\rho, u); u \in \hBall(\Map_\infty, r)\}$ is at least $r$. Using~\eqref{eq:inclusion1} together with $| \mathbb{T}_{r}| \approx r^{2a-1}$ then yields 
\[\hBall(\Map_\infty, r) \lesssim r^{2a-1}.\]
Finally, note that in terms of vertex set in the map we have $ \mathbb{T}_{r} \subset \Ball(\Map_\infty, R)$ where $R = 3+3\max_{u \in  \mathbb{T}_{r}}|\ell(u)| \approx r$. This yields a lower bound  $r^{2a-1}\lesssim| \Ball(\Map_\infty, r)| $ on the volume of balls, hence on their hull, and completes the proof of the proposition.
\end{proof}

We finally consider maps with a boundary. 

\begin{proof}[Proof of Proposition~\ref{prop:aperture}]
Suppose that $ \Map^{(p)}_{\infty}$ is constructed from a forest $ \mathcal{T}_{1}, \mathcal{T}_{2}, \dots, \mathcal{T}_{p}$ as in the preceding section where the $ \mathcal{T}_{i}, i \geq 2$ are independent two-type Galton--Watson trees and $ \mathcal{T}_{1}$ is the infinite one. Let $\Delta_{p}$ be the largest absolute value of a vertex's label belong to the finite trees $ \mathcal{T}_{2}, \dots, \mathcal{T}_{p}$. Recalling the law~\eqref{eq:increments_bridge} of the labels of the root of the trees, it follows from~\eqref{eq:convergence_LGM} that 
\[\Delta_{p} \approx p^{1/2}.\]
On the other hand, it is easy to see from the construction of $\Map_{\infty}^{(p)}$ from the forest that if $x,y$ are any two vertices on the boundary of $\Map_{\infty}^{(p)}$ then they correspond to two vertices in the forest for which the iterated chain of successors coalesce before $2 \Delta_{p}+2$ steps. Hence we deduce that 
\[\aper(\Map^{(p)}_{\infty}) \lesssim 4 \Delta_{p}+4 \approx p^{1/2}.\]
The lower bound is obtained by saying that $\aper(\Map^{(p)}_{\infty}) $ is at least the largest difference between labels of the root of the trees (they must belong to the boundary) and so of order $ p^{1/2}$.
\end{proof}

We finally prove Proposition~\ref{prop:volume_boules_cartes_bord} which we recall for the reader's convenience: for any $r$ large enough and any $p \ge 2r^{2}$, let us prove that
\[\Pr{|\Ball(\Map^{(p)}, r)| > \lambda r^{2a-1}} > c,\]
where $c > 0$ and $\lambda > 0$ are some constants which do not depend on $p$ and $r$.

\begin{proof}[Proof of Proposition~\ref{prop:volume_boules_cartes_bord}] Fix $ p \geq 2 r^{2}$. We will rely on the construction of the pointed map $\Map^{(p)}_{\bullet}$ from a forest $ \T_{1}, \dots, \T_{p}$ of i.id. two-type Galton--Watson trees together with the relation~\eqref{eq:sizebias} between $ \Map^{(p)}$ and $ \Map^{(p)}_{\bullet}$. We first assume for simplicity that the origin of $\Map^{(p)}_{\bullet}$ corresponds to the root of $ \T_{1}$ in the construction.
We will denote by $ \mathcal{E}_{\lambda}$ the following event:
\begin{enumerate}
\item The largest label in absolute value amongst the roots of $ \T_{1}, \dots, \T_{r^{2}}$ is smaller than $r/2$;

\item The maximum over $ \T_{1}, \dots, \T_{r^{2}}$ of the largest relative label in absolute value inside each tree (so each root is reset to $0$) is smaller than $r/2$ and the total number of vertices in these $r^2$ trees is larger than $\lambda r^{2a-1}$.
\item The total number of white vertices in $ \T_{1}, \dots, \T_{p}$ is less than $ p^{a- \frac{1}{2}}$.
\end{enumerate}
On the event $ \mathcal{E}_{\lambda}$ (still assuming that the origin of the map $\Map^{(p)}_{\bullet}$ is the origin of $ \T_{1}$) we have from (iii) that $| \Map^{(p)}_{\bullet}| \leq p^{a- \frac{1}{2}}$. Also, combining (i) and (ii) and using the usual bound on distances in the map we deduce that as vertex set of white vertices $ \bigcup_{1 \leq i \leq r^{2}} \T_{i} \subset \Ball( \Map^{(p)}_{\bullet}, 2r+2)$ and so the later has cardinality more than or equal to $ \lambda r^{2a-1}$. Using~\eqref{eq:sizebias} we can write 
\[\Pr{\Ball( \Map^{(p)}, 2r+2) > \lambda r^{2a-1}} 
\geq  \frac{1}{\E[1/| \Map^{(p)}_{\bullet}|]} \cdot \Es{\ind{ \mathcal{E}_{\lambda}} \cdot \frac{1}{| \Map^{(p)}_{\bullet}|}} 
\geq \frac{\P(\mathcal{E}_{\lambda})}{\E[p^{a- \frac{1}{2}}/| \Map^{(p)}_{\bullet}|]}.\]
By~\cite[Proposition 3.4]{Budd-Curien:Geometry_of_infinite_planar_maps_with_high_degrees} (and its easy extension to the case $a=5/2$, see~\cite[Eq. (51)]{Budd:The_peeling_process_of_infinite_Boltzmann_planar_maps}) we deduce that the denominator in the right-hand side is convergent and is thus bounded as $p \to \infty$. All it remains to see is that one can find $\lambda>0$ small enough so that $ \mathcal{E}_{\lambda}$ occurs with probability at least $c>0$ irrespectively of $p$ large: The first point is clearly satisfied with an asymptotically positive probability since the labels of the root of the trees converge after diffusive scaling towards a Brownian bridge see~\cite[Eq. (18)]{Le_Gall-Miermont:Scaling_limits_of_random_planar_maps_with_large_faces}. As for points (ii) and (iii), they are independent of point (i) and are clearly  satisfied with an asymptotically positive probability thanks to~\eqref{eq:convergence_LGM}. Et voil\`a. \end{proof}

{\small
\linespread{1}\selectfont 

}

\end{document}